\documentclass[a4paper, 12pt]{article}

\usepackage[latin1]{inputenc}
\usepackage[T1]{fontenc}

\usepackage[margin=2.5cm]{geometry}         
\usepackage[english]{babel}
\usepackage[markers]{endfloat}
\usepackage{mathrsfs,amsthm,amssymb,amsfonts,amsmath,color,graphicx,float,url,bbm,multirow,caption,subfig,placeins,xspace}
\usepackage{enumitem}
\usepackage[authoryear,round,sort]{natbib}

\usepackage{hyperref}
\hypersetup{pdfmenubar=true}

\addto\extrasenglish{
  
}

\numberwithin{equation}{section}

\newtheorem{lem}{Lemma}[section]

\newtheorem{pro}{Proposition}[section]
\newtheorem{theo}{Theorem}[section]
\theoremstyle{remark}
\newtheorem{rem}{Remark}

\renewcommand{\L}{\mathscr{L}}

\newcommand{\D}{\mathcal{D}}

\renewcommand{\H}{\mathcal{H}}

\newcommand{\N}{\mathbb{N}}
\renewcommand{\Im}{\mathrm{Im}}

\newcommand{\e}{\varepsilon}
\newcommand{\M}{\mathcal M_{c}}
\newcommand{\Mc}[1]{\mathcal M^{#1}_{c}}

\def\e{\varepsilon}
\def\E{\mathbb{E}}
\def\I{\mathbf{1}}
\def\R{\mathbb{R}}

\DeclareMathOperator*{\argmin}{\arg\min}

\begin{document}

	\title{Estimation via length-constrained generalized empirical principal curves  under small noise}
	\author{Sylvain Delattre $\&$ Aurélie Fischer\footnote{The research work of this author has been partially supported by ANR-18-IDEX-0001, IdEx Université de Paris.}}
	\maketitle

	\begin{abstract}
	Let $g:[0,1]\to \R^d$ be a rectifiable curve, and,	for $n\geq 1$, let $U_i^n$, $i=1,\dots,n$, denote independent random variables taking their values in $[0,1]$, with full support. Consider the model
			$$X_i^n=g(U_i^n)+\e_i^n, \quad i=1,\dots,n,$$ where $g$ is unknown and the noise tends to 0 in probability (in a way to be specified).  We are interested in the estimation of the image of $g$.
	
	Given a square integrable random vector $X$, let  $f:\left[0,1\right]\rightarrow\mathbb{R}^d$ be a minimizer of $$\Delta(f)=E\left[\min_{t\in[0,1]}V(|X-f(t)|)\right]$$ over all curves with length not greater than a certain threshold. Here, $V:\R_+\to\R_+$ is a lower semi-continuous strictly increasing function. For instance, $V(x)=x^p$, where $p>0$, or $V(x)=\frac{x}{1+x}$.

Similarly,   an empirical optimal curve associated to $X_1^n, \dots, X_n^n$ may be defined as a minimizer, under length-constraint, of the criterion
  $$\Delta_n(f)=\frac{1}{n}\sum_{i=1}^{n}\min_{t\in[0,1]}V(|X^n_i- f(t)|).$$  In this paper, we propose a method to build a sequence of generalized empirical principal curves, with selected length, so that, in Hausdorff distance, the images of the estimating principal curves  converge in probability  to the image of $g$.

	\end{abstract}


\medskip


	\section{Introduction}

	\subsection{Preliminary picture of the estimation result}

	Let $n\geq 1$. We observe random vectors $X_i^n$, given by 
	 \begin{equation}
	X_i^n=g(U_i^n)+\e_i^n, \quad i=1,\dots,n,\label{eq:mod}
	\end{equation} where the unknown function $g:[0,1]\to \R^d$ is continuous. Moreover, $g$ is assumed to have finite length equal to its 1-dimensional Hausdorff measure and to have constant speed. Here, the random variables $U^n_i$, $i=1,\dots,n$, taking their values in $[0,1]$, are independent, and belong to a class of  distributions  with full support, enclosing for instance the  uniform distribution as a particular case.

	 We study an asymptotic context, where the noise tends in probability to 0 (in a sense that will be specified below)  when the number of observations $n$ tends to infinity.
 	
	The main result of this paper is the construction, relying on the principal curve notion, of an estimator $\hat{f}_n$,  which converges  to the unknown curve $g$  in Hausdorff distance, in the sense that the Hausdorff  distance between $\Im \hat{f}_n$ and  $\Im g$ converges in probability to 0.

	\subsection{Related work}
	
	The problem of estimating the image of $g$ may be cast  into the general context of filament or manifold estimation from observations sampled on or near the unknown shape.
	
	The  literature mainly focuses on shapes with a reach bounded away from zero. The reach $\rho$, characterizing the regularity of the shape, is the maximal radius of a ball rolling on it (see \cite{Fed}).  
	In	\cite{GPVW}, an additive noise model of the form \eqref{eq:mod} is studied. The curve $g$ 	is parameterized by arc-length, normalized to $[0, 1]$. The authors assume that the $U_i$, $i=1,\dots,n$, have a common density with respect to the Lebesgue measure on $[0,1]$,  bounded and bounded away from zero. 
	The  noise has support in a ball $B(0,\sigma)$, with $\sigma <\rho(g)$, and admits a bounded density with respect to the Lebesgue measure, which is continuous on $\mathring{B}(0,\sigma)$, nondecreasing and symmetric, with a regularity condition on the boundary of the support. For an open curve (with endpoints), in addition,  $|f(1)-f(0)|/2>\sigma$. In the plane  $\R^2$, the assumptions made allow  to estimate the support $S$ of the distribution of the observations, the boundary of this set $S$, in order to find its medial axis, which  is the closure of
	the set of points in $S$ that have at least two closest points in the boundary $\partial S$. In the same article, the authors also consider clutter noise, corresponding to the situation where one observes points sampled from a mixture density $(1-\eta)u(x)+\eta h(x)$, where $u$ is the uniform density over some compact set, and $h$ is the density of points on the shape.
		Another additive model is investigated in \cite{GPVW1}, for the estimation of manifolds without boundary, with dimension lower than the dimension of the ambient space,  contained in a compact set. The model may be written 
		\begin{equation}
		X_i=G_i+\e_i, \quad i=1,\dots,n,
		\end{equation}
		where the random vectors $G_i$  are drawn uniformly on the shape $M$, and the noise is drawn  uniformly on the normal to the manifold, at distance at most $\sigma<\rho(M)$.
		The article \cite{GPVW2} is also dedicated to manifold estimation, under reach condition, first in a noiseless model, where the observations are exactly sampled on the manifold, according to some density with respect to the uniform distribution on the manifold, and then in the presence of clutter noise. An additive noise model, with known Gaussian noise, is examined as well. This latter case is related to density deconvolution.
		Estimating manifolds without boundary, with low dimension and a lower bound on the reach, is also the purpose of \cite{aamari2018,aamari2019}. The points sampled on the manifold have a common density with respect to the $d$-dimensional Hausdorff measure of the manifold, which is bounded and bounded away from zero.
	In \cite{aamari2018}, estimation relies on Tangential Delaunay Complexes. It  is performed 
	in  the noiseless case, with additive noise, bounded by $\sigma$, and under clutter noise. \cite{aamari2019} deal with compact manifolds belonging to particular regularity classes. The authors examine the noiseless situation, as well as centered bounded noise perpendicular to the manifold. Estimators based on local polynomials are proposed.
	
	To sum up, all these models involve strong conditions on the noise, which is either bounded,  or of type clutter noise.
			Such assumptions allow the authors to derive rates of convergence.
						Here, we investigate a different situation, with a weak assumption on the noise.
					 In particular, the noise does not need to be bounded. Regarding the regularity of the curve $g$, which has constant speed, there is no reach assumption, and $g$   is not required to be injective.	Although rates of convergence cannot be expected here, this weak framework is worth studying, since it is not obvious at first sight that it is even possible to build a convergent estimator without knowledge of either  length or noise.

	 The estimation strategy relies on generalized empirical principal curves.

	\subsection{Extension of the notion of length-constrained principal curve}
	
	The notion of principal curve with length constraint has been proposed by \cite{KKLZ}. According to this definition, if  $X$ denotes a random  vector with finite second moment, a principal curve is  a continuous map  $f^*:\left[0,1\right]\rightarrow\mathbb{R}^d$ minimizing under a length constraint the quantity \begin{equation}
	E\Big[\min_{t\in[0,1]}|X- f(t)|^2\Big]= E\Big[d(X,\mathrm{Im} f)^2\Big],\label{eq:pc}
	\end{equation} where $|\cdot|$ denotes the Euclidean norm and $d$ stands for the Euclidean distance to a set.
		This optimization problem may also be seen as a version of the ``average distance problem'' studied in the calculus of variations community  (see, e.g., \cite{BOS,BS03}).
Originally, a principal curve was defined by \cite{HasStuet} as a self-consistent curve, that is, a curve $f$ satisfying
$f(t_f(X))=\E[X|t_f(X)]\quad \mbox{a.s.} ,$ with $t_f$ given by $t_f(x)=\max\argmin_{t} |x-f(t)| .$  In addition to self-consistency, smoothness conditions were required: the principal curve  has to  be of class $C^\infty$, it does not intersect itself, and has finite length inside any ball in $\R^d$.  \cite{Tib} revisited the problem as a mixture model, which forces the curve $g$ in models of the form \eqref{eq:mod} to be a principal curve.
The  point of view by \cite{KKLZ}, where no  smoothness assumption is made,  was  motivated in particular by the fact that the existence of  principal curves defined in terms of self-consistency was only proved for a few particular examples (see \cite{DSextr96,DSgeo96}). Note that principal curves introduced by \cite{KKLZ} include polygonal lines.

As stated in the next lemma, shown in Section \ref{section:app-ex}, existence of optimal curves is still guaranteed when replacing the squared Euclidean distance in the definition \eqref{eq:pc}  by more general distortion measures. 

\begin{lem}\label{lem:exist}Let $V:[0,\infty)\to[0,\infty)$ be a lower semi-continuous, strictly increasing function, continuous at 0, and such that $V(0)=0$. Let $X$ denote a random vector such that $E[V(|X|)]<\infty$. Then, for any finite length $L$, there exists a curve $f^*_L:[0,1]\to\R^d$ with length $\L(f^*_L)\leq L$ minimizing over all curves with length at most $ L$ the criterion 
	$$ \Delta(f)= E\Big[V(d(X,\mathrm{Im} f))\Big].$$
	\end{lem}

The motivation for introducing this generalized  notion of principal curves is that this allows for greater flexibility in the way we measure distances. This framework encloses for instance as particular cases the power functions  $V(x)=x^p$, $p>0$. An appropriate choice of $V$ may enhance robustness. A typical example in this regard  is the function defined by $V(x)=\frac{x}{1+x}$.
	
	In a statistical context, one has at hand independent observations $X_1,\dots,X_n$, and an empirical principal curve is defined as a minimizer, under a length constraint, of the criterion   $$\frac{1}{n}\sum_{i=1}^{n}d\left(X_i,\mathrm{Im} f\right)^2.$$ Similarly, a generalized empirical principal curve may be obtained by minimizing $$\frac{1}{n}\sum_{i=1}^{n}V(d\left(X_i,\mathrm{Im} f\right)) .$$
	Observe that, in this case,  existence of a minimizer is more straightforward 
	since the empirical measure is compactly supported.

		\subsection{Organization of the paper}
		The manuscript is organized as follows. In Section \ref{section:not}, we set up notation and introduce more formally the model. In Section \ref{section:res},  we state and prove the main result: we build a sequence of generalized empirical principal curves converging  to the curve to be estimated in Hausdorff distance. The proof is structured in two subsections. The first one gathers results around the Cauchy-Crofton formula, which allows to show a useful fact about the considered class of sampling distributions on $[0,1]$.
		The proof of the existence Lemma \ref{lem:exist}, as well as a technical measurability result, are collected in Appendix \ref{section:app}.

	\section{Definitions and notation}\label{section:not}
	\subsection{Notation}
		We consider the space $(\R^d,\mathcal B(\R^d),|\cdot|)$, equipped with the standard Euclidean norm, associated to the inner product $\langle\cdot,\cdot\rangle$. Here, $\mathcal B(E)$ denotes the Borel sigma-algebra of a space $E$. 
		
	Let $\mathcal H^1$ denotes the 1-dimensional Hausdorff measure in $\R^d$.

	In the sequel, for a compact set $A$, $\mbox{diam}(A)$ stands for the diameter of a set $A$ and $d(x,A)$ for the distance from the point $x$ to the set $A$, that is $$\mbox{diam}(A)=\max_{x,y\in A} |x-y|, \quad d(x,A)=\min_{y\in A} |x-y|.$$
	We denote by $d_H(A,B)$ the Hausdorff distance between two sets $A$ and $B$, given by $$d_H(A,B)=\sup_{a\in A} d(a,B) \lor \sup_{b\in B} d(b,A).$$
	
	Let $\lambda$ stand for the Lebesgue measure and $\delta_x$ for the Dirac measure at $x$.

		Throughout, an interval $(a,b)$ will denote an open interval of $[0,1]$ equipped with the induced topology.

Denote by $\mathcal D$ a metric associated to weak convergence. For a probability measure $\mu$ and a closed set of probability measures $\mathcal M$, let $\D(\mu, \mathcal M)=\min_{\mu'\in \mathcal M}\D(\mu, \mu')$.

			For two probability measures $\mu$ and $\mu'$, we define the bounded Lipschitz metric between $\mu$ and $\mu'$  by 
		$$|\mu-\mu'|_{BL}=\sup\left\{|\mu(h)-\mu'(h)|: |h|_\infty\leq 1, \sup_{x\neq y}\frac{|h(x)-h(y)|}{|x-y|}\leq 1\right\}.$$

		A continuous function from $[0,1]$ to $\R^d$  will be called a curve. If a curve $f$ is rectifiable, its length will be denoted by $\L(f)$.
	Finally, we will denote by $C([0,1])$  the metric space of continuous functions from $[0,1]$ to $\R^d$, equipped with the topology of uniform convergence. 
	
	\subsection{Description of the model}

	Let $g:[0,1]\to \R^d$ be a  curve with finite length and constant speed, such that the length equals the 1-dimensional Hausdorff distance. 

		Given  $c>0$,  we define $\M$ as the closed family of probability distributions $\mu$ on $[0,1]$ satisfying  $\mu\geq c\lambda $ on $[0,1]$.

For $n\geq 1$, we observe a triangular array of random vectors $X_i^n$, given by the model \begin{equation}
X_i^n=g(U_i^n)+\e_i^n, \quad i=1,\dots,n,\label{eq:modtri}
\end{equation} where the $U_i^n$, $i=1,\dots,n$, are independent and for every $i=1,\dots,n$, the distribution $\mu_i^n$ of $U_i^n$ belongs to $\M$.

Let $V:\R_+\to\R_+$ be a lower semi-continuous, strictly increasing function, continuous at 0, and such that $V(0)=0$. Moreover, we assume that $V$ satisfies the following property: there exist a constant $C> 0$, such that, for every $(x,y)\in\R_+$
$$V(x+y)\leq C(V(x)+V(y)) .$$
 For a curve $f$, we define $$\Delta_n(f)=\frac{1}{n}\sum_{i=1}^{n}V(d(X_i^n,\Im f)).$$

\begin{rem}
	If we set $V(x)=x^2$, we find the usual principal curve definition by \cite{KKLZ}.
	\end{rem}

We also define a function $T(f,\cdot):\R^d\to[0,1]$, by setting $$T(f,x)=\max\argmin_{t\in[0,1]}|x-f(t)|.$$ 

For every $L>0$,
let $$G_n(L)=\min_{\L(f)\leq L}\Delta_n(f),$$
 and
let $\hat{f}_{n,L}$ denote an empirically optimal curve with length at most $L$, that is a random variable taking its values in $C([0,1])$   such that $$\Delta_n(\hat{f}_{n,L})=G_n(L).$$  Moreover, we choose $\hat f_{n,L}$ $L$-Lipschitz. We set $\Lambda_n:=\inf\{L\geq 0, G_n(L)=0\}$.

\section{Main result}\label{section:res}

We consider the estimation of the curve $g$ in Model  \eqref{eq:modtri} using a sequence of generalized empirical principal curves, that is a sequence of curves which are optimal with respect to the criterion $\Delta_n$.

\begin{theo}\label{theo:sig-unknown}
	Let $g:[0,1]\to \R^d$ be a curve, such that  $\L(g)\leq \Lambda<\infty, $ and $|g'(t)|=\L(g)$ $dt-$a.e..
Assume that  $\L(g)=\mathcal H^1(\Im g)$.
We consider Model \eqref{eq:modtri}, with   $\frac{1}{n}\sum_{i=1}^{n}V(|\e_i^n|)$ tending to 0 in probability as $n$ tends to infinity. 
Let $\hat{L}_n$ be defined by
$$ \hat{L}_n\in\argmin_{L\in a_n \mathbb N
	\cap [0,\Lambda_n\land\Lambda]}\left[L^2\mathcal D\left(\frac{1}{n}\sum_{i=1}^{n}\delta_{T(\hat{f}_{n,L},X_i^n)},\M\right)+\Delta_n(\hat{f}_{n,L})\right],$$
where $a_n>0$ for every $n\geq 1$ and $a_n\to 0$ as $n\to \infty$.
Then,
$d_H(\Im \hat{f}_{n,\hat{L}_n},\Im g)$ converges in probability to 0 as $n$ tends to infinity.

\end{theo}
Some comments are in order.

First, let us discuss the assumptions.
The requirement $\L(g)\leq \Lambda<\infty$ is technical. It allows, in the proof, to consider limit points of the constructed sequence of empirical principal curves. From a applied point of view, this is not a limitation of the procedure. Indeed, in practice, we will always consider a finite grid for the length. Moreover, with a fixed number of observations,  the minimal length needed to join all points is a finite upper bound for the length. 
The condition 
$\L(g)=\mathcal H^1(\Im g)$ ensures that the image of $g$ is parameterized with minimal possible length. Indeed, there exist an infinite number of parameterizations, with infinite possibilities for the length.  In words, generically, a portion of image of $g$ cannot be traveled several times.  The case were $g$ is injective is a particular case. Nevertheless, here, an image with loops is allowed.
We also require $|g'(t)|=\L(g)$ $dt-$a.e., which means that the image of $g$ is parameterized with constant speed $\L(g)$. These assumptions about the parametrization allow to show a key relation between the distribution class $\M$ and its image by $g$ (see Lemma \ref{lem:UH} below), the proof of which relies on the Cauchy-Crofton formula for the length of a rectifiable curve.

Observe that the main strength of the result is that it provides a convergent estimator in a very general framework. 
Neither the length, nor the noise level, converging to 0 in a very weak sense, is known. Intuitively, considering a practical situation with a fixed number of observations, the same data cloud could arise from several different generative curves, more or less long, in a model  with more or less noise. This illustrates the benefit of an estimator construction which does not require the knowledge of any of the two parameters. Apart from the upper bound $\Lambda$, which does not really need calibration in practice, as already mentioned, the procedure only  depends on a single parameter, namely the constant $c$ characterizing the class of possible sampling distributions $\M$.

It should be noticed  that the theorem does not guarantee that the procedure allows to recover the true underlying length. Nevertheless, the proof below shows that the selected length  cannot be too short: for all $\e>0$ one has $P(\hat L_n\le \L(g)-\e)\to 0$.

If $g$ is closed ($g(0)=g(1)$), then Theorem \ref{theo:sig-unknown} still holds when $\hat f_{n,L}$ is chosen as a closed empirical optimal curve with length less than $L$.

As  mentioned in the Introduction, the proof of Theorem \ref{theo:sig-unknown} is split into two parts. First, we state and prove the Cauchy-Crofton formula, together with a related result, and we use them to establish an equivalence linking $\M$ and its image by $g$ (Section \ref{sub:cauchy}). The rest of the proof of the theorem, divided in several lemmas, is presented thereafter (Section \ref{sub:proof}).

\section{Cauchy-Crofton formula and relation linking $\M$ to its image}\label{sub:cauchy}
	In the sequel, we will make use of the Cauchy-Crofton formula (\cite{Cauchy, Crofton}) for the length of a rectifiable curve $f$ in $\R^d$ (see, e.g., \cite{ayari}). We recall the formula in the next lemma, and give a proof for the sake of completeness.

Let  $\mathcal S^{d-1}=\{z\in \R^d, |z|=1\}$. For $\theta \in \mathcal S^{d-1}$ and $r\in[0,\infty)$, let  $$D_{\theta,r}=\{z\in \R^d\mid\langle\theta,z\rangle=r\}.$$
\begin{lem}[Cauchy-Crofton formula]	\label{lem:Cauchy}
	The length of a rectifiable curve $f:[0,1]\to \R^d$ is given by 
	$$\L(f)=\frac1{c_d}\int_{\mathcal S^{d-1}} \int_0^\infty \mbox{Card}(\{t\in[0,1],f(t)\in D_{\theta,r}\}) dr d\theta,$$
		where $c_d>0$ is a constant depending on the dimension $d$.
	
\end{lem}
\begin{rem}This result may also be written in the following equivalent form : 
	$$\L(f)=\frac1{c_d}\int_{\mathcal S^{d-1}} \int_0^\infty\sum_{y\in \Im f \cap D_{\theta,r}}\mbox{Card}(f^{-1}(\{y\}))drd\theta.$$\label{rem:Cauchy}
	\end{rem}

\begin{proof}
	
	For $p\geq 1$, consider the polygonal line  $f_p$ defined by the segments $[f(\frac{i}{p}),f(\frac{i+1}{p})]$, $0\leq i\leq p-1$.  We define $f_\theta$ and $f_{p,\theta}$ respectively by $f_{\theta}(t)=\langle \theta, f(t)\rangle$ and $f_{p,\theta}(t)=\langle \theta, f_p(t)\rangle$.
	There exist  $\theta_i\in \mathcal S^{(d-1)}$ and  $\rho_i\in[0,\infty)$, $0\leq i\leq p-1$, such that $f(\frac{i+1}{p})-f(\frac{i}{p})=\rho_i\theta_i$.
	Then, the variation  of $f_{p,\theta}$  is  $V({f_{p,\theta}})=\sum_{i=1}^{p-1}|\langle \theta, f(\frac{i+1}{p})-f(\frac{i}{p})\rangle|=
	\sum_{i=1}^{p-1}\rho_i|\langle \theta,\theta_i\rangle|.$ Hence, 
	$$\int_{\mathcal S^{d-1}} V({f_{p,\theta}})d\theta=
	\sum_{i=1}^{p-1}\rho_i\int_{\mathcal S^{d-1}}|\langle \theta,\theta_i\rangle|d\theta:=c_d\sum_{i=1}^{p-1}\rho_i =c_d\L(f_p).$$
	We have $\lim_{p\to+\infty}\L(f_p)=\L(f)$ and $\lim_{p\to+\infty} V({f_{p,\theta}})= V({f_{\theta}})$ \cite[Corollary of Theorem 2.1.2]{AlexRes}.
	By the Cauchy-Schwarz inequality, $V({f_{p,\theta}})\leq \L(f_p)$, and by definition of the length, $\L(f_p)\leq \L(f)$.
	Thanks to Lebesgue's dominated convergence theorem, $$\lim_{p\to+\infty}\int_{\mathcal S^{d-1}} V({f_{p,\theta}})d\theta=\int_{\mathcal S^{d-1}} V({f_{\theta}})d\theta.$$
	We deduce that $$\L(f)=\frac{1}{c_d}\int_{\mathcal S^{d-1}} V({f_{\theta}})d\theta.$$
	Besides, according to Banach's formula (see \cite{Banach}), we have
	$$ V({f_{\theta}})= \int_0^\infty \mbox{Card}(\{t\in[0,1],\langle \theta,f(t)\rangle=r\}) dr .$$
	Consequently, we get the Cauchy-Crofton formula: $$\L(f)=\frac1{c_d}\int_{\mathcal S^{d-1}} \int_0^\infty \mbox{Card}(\{t\in[0,1],\langle \theta,f(t)\rangle=r\}) dr d\theta.$$
	
\end{proof}

The next equality, corresponding  to the Cauchy-Crofton formula applied to open subset of $\Im g$, will be useful in the sequel.

\begin{rem}
	Let $(a,b)\subset[0,1]$. Then, $$\L(f|_{(a,b)})=\frac1{c_d}\int_{\mathcal S^{d-1}} \int_0^\infty \mbox{Card}(\{t\in(a,b),f(t)\in D_{\theta,r}\}) dr d\theta.$$ Since $$\L(f|_{(a,b)})=\int_0^1 \I_{(a,b)}(t) |f'(t)| dt,$$ we have $$\int_0^1 \I_{(a,b)}(t) |f'(t)| dt=\frac1{c_d}\int_{\mathcal S^{d-1}} \int_0^\infty \sum_{t\in[0,1]}\I_{(a,b)}(t)\I_{\{f(t)\in D_{\theta,r}\}} dr d\theta.$$
	Hence, by linearity, if $(a_i,b_i)$, $i\geq 1$, are pairwise disjoint open intervals of $[0,1]$, we have \begin{equation}
	\int_0^1 \I_{\bigcup_{i\geq 1} (a_i,b_i)}(t) |f'(t)| dt=\frac1{c_d}\int_{\mathcal S^{d-1}} \int_0^\infty \sum_{t\in[0,1]}\I_{\bigcup_{i\geq 1}(a_i,b_i)}(t)\I_{\{f(t)\in D_{\theta,r}\}} dr d\theta.\label{eq:IndCauchy}
	\end{equation}
	 
\end{rem}

In the sequel, we will also use a  Cauchy-Crofton-type formula for $g$ taking the form of an equality for measures.

This result relies on the following lemma.

\begin{lem}\label{lem:dominH}
	Let $f:[0,1]\to \R^d$ be a rectifiable curve. Then, the trace of $\mathcal H^1$ on $\Im f$ satisfies $\mathcal H^1\leq \gamma$, where $\gamma$ is the measure defined  on every Borel set $A\subset \Im f$ by  $$\gamma(A)= \frac 1{c_d}\int_{\mathcal S^{d-1}}\int_{0}^{\infty}\mbox{Card}( A\cap D_{\theta,r})drd\theta.$$
\end{lem}

As a preliminary result, the next lemma states the measurability of $(\theta,r)\mapsto\mbox{Card}( A\cap D_{\theta,r})$. The proof is postponed to the Appendix (Section \ref{section:app-mes}).

\begin{lem}\label{lem:mesurcard}
Let $f:[0,1]\to \R^d$ be a rectifiable curve. For $\theta\in \mathcal S^{d-1}$, $r\in \R_+$, $A$ a Borel subset of $\Im f$, let $N_{\theta,r}(A)=\mbox{Card}(A\cap D_{\theta,r})$. Then, the function $(\theta,r)\to N_{\theta,r}(A)$ is measurable for the Lebesgue sigma-algebra.
\end{lem}

\begin{proof}[Proof of Lemma \ref{lem:dominH}]
	An open subset of $\Im f$ may be written $$\mathcal O\cap\Im f=
	f\Big(\bigcup_{k\geq 1}(a_k,b_k)\Big)=\bigcup_{k\geq 1}f((a_k,b_k)),
	$$ where $\mathcal O$ is an open subset of $ \R^d$, and   $(a_k,b_k)$, $k\geq 1$, are pairwise disjoint open intervals of $[0,1]$. 
	Let $\mathcal V=\{f([\alpha,\beta]); \exists k\geq 1, a_k<\alpha<\beta<b_k, \}$. This set
	is a Vitali class for $\mathcal O\cap\Im f$, that is, for every $y=f(x)$, where $x\in \bigcup_{k\geq 1}(a_k,b_k), $ and every $\delta>0$, there exist $\alpha< \beta$, such that $y\in f([\alpha,\beta])$ and $0<\mbox{diam}(f([\alpha,\beta])) <\delta$. According to Vitali's covering theorem, for every  $\e>0$, there exist intervals $[\alpha_i,\beta_i]\subset \bigcup_{k\geq 1}(a_k,b_k)$, $i\geq 1$, such that the sets $f([\alpha_i,\beta_i])$, $i\geq 1$, are pairwise disjoint, $\mathcal H^1(\bigcup_{n\geq 1}f([\alpha_i,\beta_i]))=\mathcal{H}^1(\mathcal O\cap \Im f)$ and $\mathcal{H}^1(\mathcal O\cap\Im f)\leq \sum_{i\geq 1}\mbox{diam}(f([\alpha_i,\beta_i])) +\e$ (see \cite[Theorem 1.10]{Falc}). Hence, for every $i\geq 1$, there exist $x_i<y_i$, in  $[\alpha_i,\beta_i]$, such that 
	\begin{align*}
	\mathcal{H}^1(\mathcal O\cap\Im f)&\leq \sum_{i\geq 1}|f(x_i)-f(y_i)|+\e\\
	&=\sum_{i\geq 1}\frac 1{c_d}\int_{\mathcal S^{d-1}}\int_{0}^{\infty}\mbox{Card}(t\in[0,1],h_i(t)\in  D_{\theta,r}\})drd\theta+\e,
	\end{align*}
	thanks to the Cauchy-Crofton formula applied to the  functions $h_i:t\in[0,1]\mapsto tf(x_i)+(1-t)f(y_i)$,  for all $i\geq 1$.
	Observe that $\mbox{Card}(t\in[0,1],h_i(t)\in  D_{\theta,r}\})\in\{0,1\}$ $d\theta dr$ a.e.. By Lemma \ref{lem:mesurcard}, the function $(\theta,r)\to \mbox{Card}(A\cap D_{\theta,r})$ is measurable for the Lebesgue sigma-algebra, for every Borel subset $A$  of $\Im f$.  If $\mbox{Card}(t\in[0,1],h_i(t)\in  D_{\theta,r}\})= 1$, then $\mbox{Card}(\{f([\alpha_i,\beta_i])\cap D_{\theta,r}\})\geq 1$. Thus, 
	\begin{align*}
	\mathcal{H}^1(\mathcal O\cap\Im f)&\leq \sum_{i\geq 1}\frac 1{c_d}\int_{\mathcal S^{d-1}}\int_{0}^{\infty}\mbox{Card}(\{ f([\alpha_i,\beta_i])\cap D_{\theta,r}\})drd\theta+\e\\
	&= \frac 1{c_d}\int_{\mathcal S^{d-1}}\int_{0}^{\infty}\mbox{Card}(\{\mathcal O\cap\Im f\cap D_{\theta,r}\})drd\theta+\e.
	\end{align*}
	As $\e$ is arbitrary, $\mathcal{H}^1(\mathcal O\cap\Im f)\leq \frac 1{c_d}\int_{\mathcal S^{d-1}}\int_{0}^{\infty}\mbox{Card}(\{ \mathcal O\cap\Im f\cap D_{\theta,r}\})drd\theta$. We define $\gamma$, for every Borel set $A\subset \Im f$, by  $\gamma(A)= \frac 1{c_d}\int_{\mathcal S^{d-1}}\int_{0}^{\infty}\mbox{Card}( A\cap D_{\theta,r})drd\theta$: $\gamma$ is a measure, satisfying  $\mathcal{H}^1(\mathcal O\cap\Im f)\leq\gamma(\mathcal O\cap\Im f)$. According to the Cauchy-Crofton formula, $\gamma(\Im f)\leq \L(f)<\infty$, so that  the measure $\gamma$ is finite.  By outer regularity of finite measures,  the trace of $\mathcal H^1$ on $\Im f$ is less than the measure $\gamma$.
\end{proof}

\begin{rem}\label{rem:H1=G}For $g$ such that $\mathcal{H}^1(\Im g)=\L(g)$, $\mathcal H^1=\gamma$. Indeed, since  $\mathcal{H}^1\leq\gamma$ by Lemma \ref{lem:dominH}, it is sufficient to show that both measures have the same mass. Yet, on the one hand, $\mathcal H^1(\Im g)\leq\gamma(\Im g)$ by Lemma \ref{lem:dominH}, and on the other hand, $\gamma(\Im g)\leq \L(g)$ by the Cauchy-Crofton formula (Lemma \ref{lem:Cauchy}), so that the assumption $\mathcal{H}^1(\Im g)=\L(g)$ implies  $\mathcal{H}^1(\Im g)=\gamma(\Im g)$. 
	
\end{rem}

\begin{rem}\label{rem:card1} For $g$ such that $\mathcal{H}^1(\Im g)=\L(g)$, $\mbox{Card} (g^{-1}(\{y\}))=1$ for almost every $y$ with respect to the trace of $\mathcal H^1$ on $\Im g$. 	
This fact follows from  $$\mathcal{H}^1(\Im g)=\gamma(\Im g)=\frac 1{c_d}\int_{\mathcal S^{d-1}}\int_{0}^{\infty}\sum_ {y\in \Im g\cap D_{\theta,r}}1drd\theta,$$ together with the Cauchy-Crofton formula for $g$ (see Remark \ref{rem:Cauchy}): $$
\L(f)=\frac 1{c_d}\int_{\mathcal S^{d-1}}\int_{0}^{\infty}\sum_{y\in\Im g \cap D_{\theta,r}}\mbox{Card}(g^{-1}(\{y\}))drd\theta.
$$
\end{rem}

We are now in a position to state the next lemma, which  characterizes  the image by $g$ of a  distribution  belonging to the class $\M$.

\begin{lem}
		Let $g:[0,1]\to \R^d$ be a curve such that $0<\L(g)<\infty$, $|g'(t)|=\L(g)$ a.e., and $\mathcal H^1(\Im g)=\L(g)$. Let $\mu$ be a probability distribution supported in $[0,1]$, and let $c>0$ denote a constant.
				Then, \begin{equation}
				\mu\geq c\lambda \Leftrightarrow \forall A\subset \mathcal B(\R^d)\cap\Im g,\mu\circ g^{-1}(A)\geq c\frac{\mathcal H^1(A)}{\L(g)}.\label{eq:mu-mug-1}
				\end{equation}

\label{lem:UH}
\end{lem}

Let us denote by $\Mc{g}$  the family of probability distributions $m$ on $\R^d$, with support $\Im g$, such that $\forall A\subset \mathcal B(\R^d)\cap\Im g,m(A)\geq c\frac{\mathcal H^1(A)}{\L(g)}$. Hence, the equivalence \eqref{eq:mu-mug-1} means 
$$\mu\in\M\Leftrightarrow\mu\circ g^{-1}\in\Mc{g}.$$

In the proof of Lemma \ref{lem:UH}, we will use the fact that the property $\mathcal H^1(\Im g)=\L(g)$ may be localized, as shown in the next lemma.
\begin{lem}\label{lem:loc}Let $g:[0,1]\to \R^d$ be a curve such that $0<\L(g)<\infty$, and $\mathcal H^1(\Im g)=\L(g)$.
	Considering a subdivision $a_0=0<a_1<\dots <a_n=1$, we have, for every $1\leq i\leq n$, $$\L(g|_{(a_{i-1},a_i)})=\mathcal  H^1(g((a_{i-1},a_i))).$$
\end{lem}
\begin{proof}
	If not, there exists $i_0\in\{1,\dots,n\}$ such that $\L(g|_{(a_{i_0-1},a_{i_0})})> \mathcal H^1((a_{i_0-1},a_{i_0}))$, which implies \begin{align*}
	\mathcal	H^1(g([0,1]))&\leq \sum_{i=1}^{n} \mathcal H^1(g((a_{i-1},a_i)))\\&< \sum_{i=1}^{n}
	\L(g|_{(a_{i-1},a_i)})=\L(g).
	\end{align*}
\end{proof}

\begin{proof}[Proof of Lemma \ref{lem:loc}]
	
\begin{enumerate}
			\item[$\Rightarrow$]Assume that  $\mu\geq c\lambda$.
				 An open subset of $\Im g$ may be written $$\mathcal O\cap\Im g=
				g\left(\bigcup_{i\geq 1}(a_i,b_i)\right)=	\bigcup_{i\geq 1}g((a_i,b_i)),
				$$ where $\mathcal O$ is an open subset of $\R^d$, and  $(a_i,b_i)$, $i\geq 1$, are pairwise disjoint open intervals of $[0,1]$. 
			
			 Thanks to the  assumption on $\mu$, we have
			\begin{align}
			\mu\circ g^{-1}( \mathcal O\cap\Im g)&
			\geq c\lambda(g^{-1}(\mathcal O\cap\Im g))\nonumber\\&\geq c\lambda\left(\bigcup_{i\geq 1}(a_i,b_i)\right)
			\nonumber\\&=c\sum_{i\geq 1}(b_i-a_i)\label{eq:disj}\\&=c
			\sum_{i\geq 1}\frac{\mathcal \L(g|_{(a_i,b_i)})}{\L(g)}
			\label{eq:vitcons}\\&=c
			\sum_{i\geq 1}\frac{\mathcal H^1(g((a_i,b_i)))}{ \L( g)}
			\label{eq:L=H}\\&\geq c\frac{\mathcal H^1(\mathcal O\cap\Im g)}{ \L( g)}.\nonumber
			\end{align}
		
			For the equality \eqref{eq:disj}, we used that the intervals $(a_i,b_i)$ are disjoint, for \eqref{eq:vitcons}, the property $|g'(t)|=\L(g)$ a.e., and then for \eqref{eq:L=H}, the localized version of the equality $\mathcal H^1(\Im g)=\L(g)$ (Lemma \ref{lem:loc}).
		
		 The result extends to every Borel subset of $\Im g$, using the outer regularity of probability measures.

	\begin{rem}\label{rem:h=lambda}
	 	 Taking $c=1$ and $\mu=\lambda$, we obtain that $\lambda\circ g^{-1}$ is the trace of $\frac{\H^1}{\L(g)}$ on $\Im g$, since both measures are probability measures.
	 \end{rem}

		\item[$\Leftarrow$] Assume that  $\forall A\subset \mathcal B(\R^d)\cap\Im g,\mu\circ g^{-1}(A)\geq c\frac{\mathcal H^1(A)}{\L(g)}$. An open subset of $[0,1]$, for the induced topology, has the form $\bigcup_{i\geq 1}(a_i,b_i)$, where $(a_i,b_i)$, $i\geq 1$ are pairwise disjoint open intervals of $[0,1]$. Let $\mathcal O\cap \Im g=g\left(\bigcup_{i\geq 1}(a_i,b_i)\right)$. Using the  assumption, the fact that $\mathcal H^1=\gamma$ (Remark \ref{rem:H1=G}), and the property $\mbox{Card}(g^{-1}(\{y\}))=1$ for a.e. $y$ with respect to the trace of $\mathcal H^1$ on $\Im g$ (Remark \ref{rem:card1}), we may write
		\begin{align*}
	\mu\circ g^{-1}(\mathcal O\cap\Im g)&\geq c\frac{\mathcal H^1(\mathcal O\cap\Im g)}{\L(g)}
		\\&= \frac{c}{\L(g)}\frac 1{c_d}\int_{\mathcal S^{d-1}}\int_{0}^{\infty}\mbox{Card}(\mathcal O\cap\Im g\cap D_{\theta,r})drd\theta
				\\&= \frac{c}{\L(g)}\frac 1{c_d}\int_{\mathcal S^{d-1}}\int_{0}^{\infty}\sum_ {y\in \mathcal O\cap\Im f\cap D_{\theta,r}}1drd\theta
					\\&= \frac{c}{\L(g)}\frac 1{c_d}\int_{\mathcal S^{d-1}}\int_{0}^{\infty}\sum_ {y\in \mathcal O\cap\Im f\cap D_{\theta,r}}\mbox{Card}(g^{-1}(\{y\}))drd\theta
		\\&=\frac{c}{\L(g)}\frac 1{c_d}\int_{\mathcal S^{d-1}}\int_{0}^{\infty} \sum_{t\in[0,1]}\I_{\bigcup_{i\geq 1}(a_i,b_i)}(t)\I_{\{g(t)\in D_{\theta,r}\}} dr d\theta\end{align*}
		Thanks to the equality \eqref{eq:IndCauchy}, and using $|g'(t)|=\L(g)$ a.e., we deduce that
			\begin{align*}
		\mu\circ g^{-1}(\mathcal O\cap\Im g)	&\geq\frac{c}{\L(g)}\int_0^1 \I_{\bigcup_{i\geq 1} (a_i,b_i)}(t) |g'(t)| dt
			\\&=c\sum_{i\geq 1}(b_i-a_i)\\&=c\lambda\bigg(\bigcup_{i\geq 1}(a_i,b_i)\bigg).
				\end{align*}

Let us show that $\{t\in[0,1], \mbox{Card} (g^{-1}(\{g(t)\}))>1\}$ is negligible for $\lambda$.

Let $A\subset \Im g$ be a negligible set for the trace of $\mathcal H^1$ on $\Im g$. Then, $g^{-1}(A)$ is negligible for $\lambda$. Indeed,  there exists a Borel set $N$, such that $A\subset N$ and $\H^1(N)=0$. Since $g^{-1}(A)\subset g^{-1}(N)$, $\lambda(g^{-1}(A))\leq  \lambda(g^{-1}(N)).$
By Remark \ref{rem:h=lambda}, $\lambda(g^{-1}(N))= \frac{\H^1(N)}{\L(g)}$. Thus, $\lambda(g^{-1}(A))=0$.

Hence, the fact that $\{y\in\Im g, \mbox{Card} (g^{-1}(\{y\}))>1\}$ is a negligible set for the trace of $\mathcal H^1$ on $\Im g$ (Remark \ref{rem:card1}) implies that $\{t\in[0,1], \mbox{Card} (g^{-1}(\{g(t)\}))>1\}$ is negligible for $\lambda$.

 Consequently, 
	\begin{align*}
	\mu\bigg(\bigcup_{i\geq 1}(a_i,b_i)\bigg)
&=
	\mu\circ g^{-1}(\mathcal O\cap\Im g),
	\end{align*}
	and, thus, 
	$$\mu\bigg(\bigcup_{i\geq 1}(a_i,b_i)\bigg)\geq c\lambda\bigg(\bigcup_{i\geq 1}(a_i,b_i)\bigg).$$	This extends to every Borel subset of $[0,1]$ by outer regularity. Hence, $\mu\geq c\lambda$.
	\end{enumerate}

\end{proof}

Now, equipped with Lemma \ref{lem:UH}, let us turn to the proof of Theorem \ref{theo:sig-unknown} itself.

\section{Proof of Theorem \ref{theo:sig-unknown}}\label{sub:proof}

Note that $\L(g)-a_n<a_n\left\lfloor\frac{\L(g)}{a_n}\right\rfloor\leq \L(g)$. We set, for every $n\geq 1$, $f^*_n:=\hat{f}_{n,a_n\left\lfloor\frac{\L(g)}{a_n}\right\rfloor}$.

\subsection{Step 1}

	We will first  prove that $L^2\mathcal D\left(\frac{1}{n}\sum_{i=1}^{n}\delta_{T(f^*_n,X_i^n)},\M\right)+\Delta_n(f^*_n)$ converges in probability to 0 as $n$ goes to infinity.
	
	To begin with, let us consider the term $\Delta_n(f^*_n)$.
	\begin{lem}$\Delta_n(f^*_n)$ converges in probability to 0 as $n$ tends to infinity.
		\label{lem:Deltan->0}
	\end{lem}
	
	\begin{proof}[Proof of Lemma \ref{lem:Deltan->0}]
		We write \begin{align*}
		\Delta_n(f^*_n)&\leq \Delta_n(g)\\&=\frac{1}{n}\sum_{i=1}^{n}V(d(X_i^n, \Im g))=\frac{1}{n}\sum_{i=1}^{n}V(d(g(U_i^n)+\e_i^n, \Im g))\\&\leq
		\frac{C}{n}\sum_{i=1}^{n}V(d(g(U_i^n), \Im g))+\frac{C}{n}\sum_{i=1}^{n}V(|\e_i^n|)=
		\frac{C}{n}\sum_{i=1}^{n}V(|\e_i^n|),
		\end{align*}which converges in probability to 0. 
	\end{proof}

	For $n\geq 1 $, let $\D_n^*(\M)=\D\left(\frac{1}{n}\sum_{i=1}^{n}\delta_{T(f^*_n,X_i^n)},\M\right)$. The remaining part of Step 1 is dedicated to the convergence of $(\D_n^*(\M))_{n\geq 1}$.

	\begin{pro}
		\label{prop:Dn->0}
	If $\L(g)>0$, the sequence	$(\D_n^*(\M))_{n\geq 1}$ converges in probability to 0 as $n$ tends to infinity.
		\end {pro}
		\subsubsection*{Proof of Proposition \ref{prop:Dn->0}}
	We will show that the sequence $(\D_n^*(\M))_{n\geq 1}$ is tight and that every limit point for the convergence in distribution  is $\delta_0$.
		We set $\nu^*_n=\frac 1 n \sum_{i=1}^{n}\delta_{\big(T(f^*_n,X_i^n),X_i^n\big)}$ and ${\nu^*_n}'=\frac 1 n \sum_{i=1}^{n}\delta_{\big(T(f^*_n,X_i^n),g(U_i^n)\big)}$.
	
	\begin{lem}\label{lem:nun'nun}
The difference $|\nu^*_n-{\nu^*_n}'|_{BL}$ converges in probability to 0 as $n$ tends to infinity.
	\end{lem}

\begin{proof}[Proof of Lemma \ref{lem:nun'nun}]We have
	\begin{align}
	&|\nu^*_n-{\nu^*_n}'|_{BL}\nonumber\\&=\sup\left\{|\nu^*_n(h)-{\nu^*_n}'(h)|: |h|_\infty\leq 1, \sup_{x\neq y}\frac{|h(x)-h(y)|}{|x-y|}\leq 1\right\}\nonumber\\&\leq
	\sup\left\{\frac 1 n\sum_{i=1}^{n}|h(T(f^*_n,X_i^n),X_i^n)-h(T(f^*_n,X_i^n),g(U_i^n))|: |h|_\infty\leq 1, \sup_{x\neq y}\frac{|h(x)-h(y)|}{|x-y|}\leq 1\right\}\nonumber\\&\leq \frac 1 n \sum_{i=1}^{n}|\e_i^n|\nonumber\\&\leq \left(\frac 1 n \sum_{i=1}^{n}|\e_i^n|^2\right)^{1/2},\label{eq:nun}
	\end{align}
	which converges in probability  to 0 as $n$ tends to infinity.
\end{proof}
	
\begin{lem}\label{lem:fn*nu'tight}
	The sequence $\left(f^*_n,{\nu^*_n}'\right)_{n\geq 1}$ is tight.
	\end{lem}

\begin{proof}[Proof of Lemma \ref{lem:fn*nu'tight}]
	Observe that, for every $n\geq 1$, ${\nu^*_n}'$ is supported in a compact set $[0,1]\times \Im g$, and thus, belongs to a compact set of measures for the topology of weak convergence, so that the sequence $({\nu^*_n}')_{n\geq 1}$ is tight.

	Besides, for every $n\geq 1$, $f^*_n$ is $\L(g)$--Lipschitz. Thus,  to obtain tightness of the sequence $(f^*_n)_{n\geq 1}$ in $C([0,1])$, it suffices to show that the sequence $(f^*_n(0))_{n\geq 1}$ is tight (see, e.g.,  \cite[Lemma 2.1]{Prokh}). If it is not the  case, there exists $\eta>0$ such that, for every $A>0$, there exists $n\geq 1$ such that  $P(|f^*_n(0)|>A)\geq\eta>0$, and hence, as the length $\L(g)$ is finite,   \begin{equation}
	P(d(0,\Im f^*_n)>A-\L(g))\geq\eta>0.\label{eq:mintight}
	\end{equation}
	Observe that $$\Delta_n(f^*_n)\leq \Delta_n(0)=\frac 1 n \sum_{i=1}^{n}V(|X_i^n|).$$
	Hence, since the sequence $\left(\frac 1 n \sum_{i=1}^{n}V(|X_i^n|)\right)_{n\geq 1}$ is tight, for every $\e>0$, there exists $M_\e$ such that for all $n\geq 1$, \begin{equation}
	P(\Delta_n(f^*_n)>M_\e)\leq P\left(\frac 1 n \sum_{i=1}^{n}V(|X_i^n|)>M_\e\right)\leq \e.\label{eq:borneX}
	\end{equation}
	
	Moreover, \begin{align*}
	P\left(\Delta_n(f^*_n)>M_\e\right)&=P\left(\frac{1}{n}\sum_{i=1}^{n}V(d(X_i^n,\Im f^*_n))>M_\e\right)
	\\&\geq
	P\left(\min_{t\in[0,1]}\frac{V(|f^*_n(t)|)}{C}-\frac{1}{n}\sum_{i=1}^{n}V(|X_i^n|)>M_\e\right)\\&\geq
	P\left(V(d(0,\Im f^*_n))>2CM_\e\right)-P\left(\frac{1}{n}\sum_{i=1}^{n}V(|X_i^n|)>M_\e\right)\\&\geq P\left(V(d(0,\Im f^*_n))>2CM_\e\right)-\e,
	\end{align*}according to \eqref{eq:borneX}. Thus, thanks to \eqref{eq:mintight}, there exists $n_\e$ such that 
	$$\eta-\e\leq P\left(\Delta_{n_\e}(f^*_{n_\e})>M_\e\right)\leq \e,$$ which leads to a contradiction since $\e$ is arbitrary.
	Consequently, $(f^*_n(0))_{n\geq 1}$ is tight.
\end{proof}

From  Lemma \ref{lem:fn*nu'tight}, the sequence
$(\D_n^*(\M))_{n\geq 1}$ is tight, and thus, by Prohorov's theorem, $(\D_n^*(\M))_{n\geq 1}$ has a limit point for weak convergence. So, there exists a weakly convergent  subsequence. Let us consider an arbitrary such convergent subsequence $(\D_{\sigma(n)}^*(\M))_{n\geq 1}$,  where $\sigma:\N\to\N$ denotes a strictly increasing function, and show that the corresponding limit point  is $\delta_0$. The sequence $\left(f^*_n,{\nu^*_n}\right)_{n\geq 1}$ is tight, by Lemma \ref{lem:nun'nun} and Lemma \ref{lem:fn*nu'tight}. So, by Prohorov's theorem, we may assume, up to extracting  a further subsequence, that the sequence  
$\big(f^*_{\sigma(n)},\nu^*_{\sigma(n)}\big)_{n\geq 1}$ converges in distribution to a tuple $(\varphi^*,\nu^*)$, where $\varphi^*$ is a random function, and $\nu^*$ a random probability measure on $[0,1]\times \R^d$. 
According to Skorokhod's representation theorem, there exist a random sequence  $(\tilde{f}^*_{\sigma(n)}, \tilde{\nu}^*_{\sigma(n)})_{n\geq 1}$ and a tuple $(\tilde{\varphi}^*, \tilde{\nu}^*)$ with the same distribution as  $(f^*_{\sigma(n)},\nu^*_{\sigma(n)})_{n\geq 1}$ and  $(\varphi^*,\nu^*)$ respectively, such that $(\tilde{f}^*_{\sigma(n)}, \tilde{\nu}^*_{\sigma(n)})_{n\geq 1}$ converges almost surely to $(\tilde{\varphi}^*, \tilde{\nu}^*)$.
Moreover, up to considering an extension of the probability space where this representation holds, there exist random vectors $\tilde{X}_i^n$, $i=1,\dots, n$, such that  $\big((\tilde{X}_i^{\sigma(n)})_{1\leq i\leq \sigma(n)},\tilde{f}^*_{\sigma(n)}\big)$ has the same joint distribution as $\big((X_i^{\sigma(n)})_{1\leq i\leq \sigma(n)},f^*_{\sigma(n)}\big)$. In the sequel, to lighten notation, the tilde will be omitted. The marginal distributions of $\nu^*$ will be denoted by $\nu^{*,i}$, $i=1,2$.

\begin{lem}
	Assume that $\L(g)>0$.
Then,	we have \label{lem:nu2M}$\nu^{*,2}\in \Mc{g}$ a.s.
\end{lem}
  \begin{proof}[Proof of Lemma \ref{lem:nu2M}]
   
  	 First, recall that $\nu_{\sigma(n)}^*$ converges almost surely to $\nu^*$. This implies that  $\frac 1 {\sigma(n)}\sum_{i=1}^{\sigma(n)}\delta_{g(U_i^{\sigma(n)})}$ converges almost surely to $\nu^{*,2}$.
  	 
  	 For $n\geq 1$, the measure $\frac 1 n \sum_{i=1}^{n}\mu_i^n\circ g^{-1}$  is supported in the compact set $\Im g$, so that 	$\left(\frac 1 n \sum_{i=1}^{n}\mu_i^n\circ g^{-1}\right)_{n\geq 1}$  is tight. Thus, up to a further extraction, we may assume that $\frac 1 {\sigma(n)} \sum_{i=1}^{{\sigma(n)}}\mu_i^{\sigma(n)}\circ g^{-1}$ converges weakly to a measure $m$.

Besides, for every continuous and bounded function $h$, by the weak law of large numbers for triangular arrays, we have $\frac 1 n \sum_{i=1}^n h(g(U_i^n)) - \frac 1 n \sum_{i=1}^n E[ h(g(U_i^n))] \to0$ in probability almost surely. Consequently, for every $x\in\mathbb R^d$,  
$\int \exp(i\langle u,x\rangle)d\nu^{*,2}(u)=\int \exp(i\langle u,x\rangle)dm(u)$ almost surely. Hence, almost surely, for every $x\in\mathbb Q^d$,
$$\int \exp(i\langle u,x\rangle)d\nu^{*,2}(u)=\int \exp(i\langle u,x\rangle)dm(u).$$

 We obtain that $\nu^{2,*}=m$ a.s..
 
  By Lemma \ref{lem:UH}, for $i=1,\dots,n$, $\mu_i^n\circ g^{-1}\in \Mc{g}$. As $ \Mc{g}$ is convex, for every $n\geq 1$,  $\frac 1 n \sum_{i=1}^{n}\mu_i^n\circ g^{-1}\in \Mc{g}$. Since $\Mc{g}$ is closed, $\nu^{2,*}\in\Mc{g}$ a.s..

 	 \end{proof}

\begin{lem}\label{lem:mint}
	Denoting by $(T,Z)$ the identity on $[0,1]\times \R^d$,  we have \begin{equation*}
\nu^*\big(	|Z-\varphi^*(T)|=\min_{t\in[0,1]}|Z-\varphi^*(t)|\big)=1 \mbox{ a.s.}\label{eq:mint}
	\end{equation*}
\end{lem}

\begin{proof}[Proof of Lemma \ref{lem:mint}]
	For a curve $f$, define, for $\e>0$, the open set
	$$\mathcal O^\e_f=\{|Z-f(T)|>\min_{t\in[0,1]}|Z-f(t)|+\e\}.$$ 
	By the Portmanteau lemma,
	$\nu^*(\mathcal O^\e_{\varphi^*})\leq \liminf \nu^*_{\sigma(n)}(\mathcal O^\e_{\varphi^*})$ a.s., and 
	by definition of $f_{\sigma(n)}^*$, for every $\delta>0$, $\nu^*_{\sigma(n)}\big(\mathcal O^\delta_{f^*_{\sigma(n)}}\big)=0$.
	Moreover, thanks to the convergence of $f_{\sigma(n)}^*$ to $\varphi^*$, $\mathcal O^\e_{\varphi^*}\subset \mathcal O_{f^*_{\sigma(n)}}^{\e/2}$ a.s. as soon as  $n$ is large enough. Thus, $\nu^*(\mathcal O^\e_{\varphi^*})=0$ a.s., and, letting $\e$ tend to 0, we get $$\nu^*\big(|Z-\varphi^*(T)|>\min_{t\in[0,1]}|Z-\varphi^*(t)|\big)=0\mbox{ a.s.},$$ that is $$\nu^*\big(|Z-\varphi^*(T)|=\min_{t\in[0,1]}|Z-\varphi^*(t)|\big)=1\mbox{ a.s.},$$ as desired.
\end{proof}

\begin{lem}\label{lem:delta*=0}
	We have $\int |z-\varphi^*(t)|^2d\nu^*(t,z)=0$ almost surely.
\end{lem}

\begin{proof}[Proof of Lemma \ref{lem:delta*=0}]
	First, recall that $\Delta_n(f^*_n)$ converges in probability to 0 as $n$ tends to infinity by Lemma \ref{lem:Deltan->0}.
	Let $$\Delta_n'(f)=\frac{1}{n}\sum_{i=1}^{n}V(d(g(U_i^n),\Im f)).$$
	 We have \begin{align*}
	\Delta_n'(f^*_n)&\leq \frac{C}{n}\sum_{i=1}^{n}V(d(X_i^n,\Im f_n^*))
+\frac Cn\sum_{i=1}^{n}V(|\e_i^n|)\\&=C\Delta_n(f_n^*)+\frac Cn\sum_{i=1}^{n}V(|\e_i^n|).
	\end{align*} Thus, $\Delta'_n(f^*_n)$  converges in probability to 0, thanks to the convergence of $\Delta_n(f^*_n)$ and the assumption on the $\e_i^n$.
	Moreover, 
	\begin{align*}
	&\frac{1}{n}\sum_{i=1}^{n}V(|g(U_i^n)-\varphi^*(T(f^*_n,X_i^n))|) \\& =
	\frac{1}{{n}}\sum_{i=1}^{{n}}V\big(|g(U_i^{n})-f^*_n(T(f^*_{n},X_i^{n}))+f^*_n(T(f^*_{n},X_i^{n}))-\varphi^*(T(f^*_{n},X_i^{n}))|\big)
	\\&\leq \frac C n \sum_{i=1}^nV(d(g(U_i^n),\Im f_n^*))
+\frac{C}{n}\sum_{i=1}^{n} V(|f^*_n(T(f^*_n,X_i^n))-\varphi^*(T(f^*_n,X_i^n))|),
	\end{align*} 
	so that $\frac{1}{{\sigma(n)}}\sum_{i=1}^{{\sigma(n)}}V(|g(U_i^{\sigma(n)})-\varphi^*(T(f^*_{\sigma(n)},X_i^{\sigma(n)})|)$ converges in probability to 0, using the convergence of $\Delta_n'(f^*_n)$,  the uniform
	 convergence of $f^*_{\sigma(n)}$ to $\varphi^*$, and the fact that $V$ is continuous, with $V(0)=0$.
	The function $\psi$ defined by $\psi(z,t)=|z-\varphi^*(t)|$ for $(z,t)\in \Im g\times [0,1]$ is lower semi-continuous.
	This implies that, almost surely,
	  $$\liminf_{n\to \infty}\frac{1}{\sigma(n)}\sum_{i=1}^{\sigma(n)}V(|g(U_i^{\sigma(n)})-\varphi^*( T(f^*_{\sigma(n)},X_i^{\sigma(n)}))|)\geq  \int V(|z-\varphi^*(t)|)d\nu^*(t,z).$$
	Consequently, $\int V(|z-\varphi^*(t)|)d\nu^*(t,z)=0$ a.s.
\end{proof}

Note that $\L(\varphi^*)\leq \L(g)$ a.s., by the lower semi-continuity property of the length (see, e.g., \citet[Theorem 2.1.2]{AlexRes}). Together with Lemma \ref{lem:delta*=0}, this property allows to show the next result.
\begin{lem}\label{lem:ImL=}
Assume $\L(g)>0$.	We have $\Im g= \Im\varphi^*$ a.s. and $\L(\varphi^*)=\L(g)$ a.s..
\end{lem}

\begin{proof}[Proof of Lemma \ref{lem:ImL=}]
	Since, by Lemma \ref{lem:delta*=0}, $\int |z-\varphi^*(t)|^2d\nu^*(t,z)=0$ almost surely,  $\nu^*(Z=\varphi^*(T))=1$ a.s., that is, by Lemma \ref{lem:mint},
	$\nu^{*}(d(Z,\Im\varphi^*)=0)=1$ a.s..
As $\nu^{*,2}\in\Mc{g}$,	$d(g(u),\Im\varphi^*)=0$ $du-$a.e., a.s.. By continuity, $d(g(t),\Im\varphi^*)=0$ for every $t\in[0,1]$, a.s..
	Thus,  $\Im g\subset \Im\varphi^*$ a.s..
	As $\L(\varphi^*)\leq \L(g)$ a.s., using the assumption $\L(g)=\mathcal H^1(\Im g)$, we have in fact equality $\Im g= \Im\varphi^*$ a.s., and $\L(\varphi^*)=\L(g)$ a.s..
\end{proof}

In particular, $\L(\varphi^*)=\mathcal H^1(\Im \varphi^*)$ a.s..
Moreover, since $f^*_n$ is $\L(g)$--Lipschitz for every $n\geq 1$, passing to the limit for $f^*_{\sigma(n)}$, we obtain that $\varphi^*$ is $\L(g)$--Lipschitz, that is $\L(\varphi^*)$--Lipschitz, a.s..
We deduce that $|{\varphi^*}'(t)|=\L(\varphi^*)$ $dt-$a.e., a.s..
So, $\varphi^*$ satisfies the assumptions of Lemma \ref{lem:UH}.
	
\begin{lem}\label{lem:cv1}
Assume that $\L(g)>0$. We have $$\D_{\sigma(n)}^*(\M)\to 0,$$
in probability.
\end{lem}

\begin{proof}[Proof of Lemma \ref{lem:cv1}]

	The distribution $\nu^{*,1}\circ(\varphi^*)^{-1}$ is  $\nu^{*,2}$ a.s.. Since $\nu^{*,2}$ belongs to $\Mc{g}$ a.s., which is $\Mc{\varphi^*}$ a.s., Lemma \ref{lem:UH} shows that
	$\nu^{*,1}\in\M$ a.s..

	Hence, $$\D\left(\frac{1}{\sigma(n)}\sum_{i=1}^{\sigma(n)}\delta_{T(f^*_{\sigma(n)},X_i^{\sigma(n)})},\M\right)$$ tends to 0 in probability as $n$ tends to infinity.
\end{proof}

Finally, $(\D_{n}^*(\M))_{n\geq 1}$ converges in probability to 0, which proves Proposition \ref{prop:Dn->0}.

\subsection{Step 2}

Let $\hat{f}_n:=\hat{f}_{n,\hat L_n}$.  In the sequel, we will consider extractions of the sequence $(\hat{f}_n)_{n\geq 1}$ converging in distribution, and show that for every limit point $\varphi$ of   $(\hat{f}_n)_{n\geq 1}$, $d_H(\Im \varphi,\Im g)=0$ a.s..

By definition of $\hat{L}_n$, we have $$\hat L_n^2 \mathcal D\left(
\frac{1}{n}\sum_{i=1}^{n}\delta_{T(\hat f_n,X_i^n)} ,\M\right)+\Delta_n(\hat{f}_n)\leq \L^2(f_n^*)\mathcal D\left(\frac{1}{n}\sum_{i=1}^{n}\delta_{T(f^*_n,X_i^n)},\M\right)+\Delta_n(f_n^*).$$ By Lemma \ref{lem:Deltan->0} and Proposition \ref{prop:Dn->0}, the right-hand term tends to 0, which implies the two following results.

\begin{lem}\label{lem:Deltanhat->0}
$\Delta_n(\hat f_n)$ converges in probability to 0 as $n$ tends to infinity.
\end{lem}

\begin{lem}\label{lem:Dnhat->0} We have
\begin{equation*}
\hat L_n^2 \mathcal D\left(\frac{1}{n}\sum_{i=1}^{n}\delta_{T(\hat f_n,X_i^n)} ,\M\right)\to 0,
\end{equation*}in probability.
\end{lem}

 We set $\nu_n=\frac 1 n \sum_{i=1}^{n}\delta_{\big(T(\hat f_n,X_i^n),X_i^n\big)}$ and $\nu'_n=\frac 1 n \sum_{i=1}^{n}\delta_{\big(T(\hat f_n,X_i^n),g(U_i^n)\big)}$.
 
 The same arguments as in the proof of Lemma \ref{lem:nun'nun} and Lemma \ref{lem:fn*nu'tight} lead to the following statements for $\nu_n$, $\nu'_n$, and $\hat f_n$. 

	\begin{lem}\label{lem:mun'mun}
	The difference $|\nu_n-\nu'_n|_{BL}$ converges in probability to 0 as $n$ tends to infinity.
\end{lem}

\begin{lem}\label{lem:fnmu'tight}
	The sequence $(\hat f_n,{\nu_n'})_{n\geq 1}$ is tight.
\end{lem}

To obtain this result, it suffices to notice that  $\hat{f}_n$ is $\hat{L}_n-$Lipschitz for every $n\geq 1$, and, thus, $\Lambda-$Lipschitz, and to show that the sequence $(\hat f_n(0))_{n\geq 1}$ is tight, by replacing $\L(g)$ by $\Lambda$ in the proof of Lemma \ref{lem:fn*nu'tight}.

By  Lemma \ref{lem:fnmu'tight}, the sequence $(d_H(\Im \hat{f}_n,\Im g))_{n\geq 1}$ is tight.  Considering an arbitrary limit point for weak convergence, we let $\kappa:\N\to\N$ denote a strictly increasing function such that the subsequence $(d_H(\Im \hat{f}_{\kappa(n)},\Im g))_{n\geq 1}$ is weakly convergent. We will show that the considered limit point is $\delta_0$. By Lemma  \ref{lem:mun'mun} and Lemma \ref{lem:fnmu'tight}, the sequence $(\hat f_n,\nu_n)_{n\geq 1}$ is tight. Thus, by Prohorov's theorem, up to an extraction, we may assume that $(\hat f_{\kappa(n)},\nu_{\kappa(n)})_{n\geq 1}$ converges in distribution to a tuple $(\varphi,\nu)$, where $\varphi$ is a random function and $\nu$ a random probability measure on $[0,1]\times \R^d$. Thanks to Skorokhod's representation theorem, 
there exist a random sequence  $(\tilde{\hat f}_{\kappa(n)}, \tilde{\nu}_{\kappa(n)})_{n\geq 1}$ and a tuple $(\tilde{\varphi}, \tilde{\nu})$ with the same distribution as  $(\hat f_{\kappa(n)},\nu_{\kappa(n)})_{n\geq 1}$ and  $(\varphi,\nu)$ respectively, such that $(\tilde{\hat f}_{\kappa(n)}, \tilde{\nu}_{\kappa(n)})_{n\geq 1}$ converges almost surely to $(\tilde{\varphi}, \tilde{\nu})$.
Again, the tilde will be omitted to lighten notation, and the marginal distributions of $\nu$ will be denoted by $\nu^{i}$, $i=1,2$. 
 As above, $\nu^2\in\Mc{g}$.
  Moreover, we have the next result, similar to  Lemma \ref{lem:mint}.
\begin{lem}
	  Denoting by $(T,Z)$ the identity  on $[0,1]\times \R^d$, we have \begin{equation}
\nu\big(|Z-\varphi(T)|=\min_{t\in[0,1]}|Z-\varphi(t)|\big)=1 \mbox{ a.s.}\label{lem:mint2}
\end{equation}
\end{lem}

The next lemma is obtained by mimicking the proof of Lemma \ref{lem:delta*=0}. 
\begin{lem}\label{lem:delta=0}
	We have $\int |z-\varphi(t)|^2d\nu(t,z)=0$ almost surely.
\end{lem}

	\begin{lem}\label{lem:Img=Imphi}
		We have  $\Im g=  \Im\varphi$ almost surely.
	\end{lem}

  \begin{proof}[Proof of Lemma \ref{lem:Img=Imphi}]
 	Since, by Lemma \ref{lem:delta=0}, $\nu( Z=\varphi(T))=1$ almost surely, using Lemma \ref{lem:mint2} and the fact that $\nu^2\in\Mc{g}$, we get, as in the proof of Lemma \ref{lem:ImL=},   $\Im g\subset \Im \varphi$ a.s..
	On the event $\{\L(\varphi)=0\}$, we get $\Im g=\Im \varphi$ a.s..
From now on, we consider the event  $\{\L(\varphi)>0\cap \Im g\neq  \Im\varphi\}$. It suffices to prove that it has probability zero. Yet, on this event,  $\liminf \hat L_n \ge \L(\varphi)>0$ a.s., therefore Lemma \ref{lem:Dnhat->0} yields the convergence in probability
	$$\mathcal D\left(\frac{1}{n}\sum_{i=1}^{n}\delta_{T(\hat f_n,X_i^n)} ,\M\right)\to 0,$$ which implies $\nu^1\in\M$ a.s..
	Moreover, there exists $t\in[0,1]$ such that $\varphi(t)\notin \Im g$, that is $d(\varphi(t),\Im g)>0$. By continuity,  $d(\varphi(s),\Im g)>0$ for every $s$ in a non-empty open interval $(a,b)$.
	We get $\nu(\varphi(T) \notin\Im g)\ge \nu( T\in(a,b)) \ge c(b-a)$ since $\nu^1\in\M$ a.s..
	Besides, $\nu( Z=\varphi(T))=1$  and $\nu(Z\in\Im g)=1$ a.s., leading to a contradiction whenever the event $\{\L(\varphi)>0\cap \Im g\neq  \Im\varphi\}$ has positive probability.
	So, $\Im g=  \Im\varphi$ a.s..
 \end{proof}

Hence, $d_H(\Im \varphi,\Im g)=0$ a.s., as desired. In other words, every limit point of the sequence $(d_H(\Im \hat{f}_n,\Im g))_{n\ge 1}$ is $\delta_0$, and, thus, Theorem \ref{theo:sig-unknown} is proved.

	\appendix
\section{Appendix}	

\label{section:app}

\subsection{Proof of Lemma \ref{lem:exist}}
\label{section:app-ex}

First, we show that  $\Delta$ is lower semi-continuous. Let $(f_n)_{n\geq 1}$ be a sequence of curves converging uniformly to a curve $f$. Then, for every $x\in\R^d$, $d(x,\Im f_n)$ converges to $d(x,\Im f)$, since $|d(x,\Im f_n)- d(x,\Im f)|\leq \sup_t |f(t)-f_n(t)|$. Then, by lower semi-continuity of $V$, for every $x\in\R^d$, \begin{equation}\label{eq:Vlsc}
	\liminf_{n\to \infty} V(d(x,\Im f_n))\geq V(d(x,\Im f)).
	\end{equation}
	By Fatou's Lemma, $$\liminf_{n\to \infty}\Delta(f_n)=\liminf_{n\to \infty} E[V(d(x,\Im f_n))]\geq E\left[ \liminf_{n\to \infty}V(d(x,\Im f_n))\right].$$ Using \eqref{eq:Vlsc}, we obtain $$\liminf_{n\to \infty}\Delta(f_n)\geq
	E[ V(d(x,\Im f))]=\Delta(f).$$ So, $\Delta$ is lower semi-continuous.
	
	Now, we prove that a minimizing sequence of $\Delta$ is relatively compact.
	Let $(f_n)_{n\geq 1}$ denote a  sequence  of curves with length at most $L$ and constant speed,  which is a minimizing sequence, that is $$\lim_{n\to\infty}\Delta(f_n)=\inf_{f,\L(f)\leq L}\Delta(f).$$
	
	For every $n\geq 1$, $f_n$ is $L$-Lipschitz. Thus, the sequence is equi-uniformly continuous. Let us show  that $(f_n(0))_{n\geq 1}$ is bounded. We may write, for every $n\geq 1$, $t\in[0,1]$, $|f_n(t)| \ge |f_n(0)| - L t\ge |f_n(0)| - L$. Here, the length $L$ is finite. 	
	Thus, if there exists a strictly increasing function $\kappa:\N\to \N$ such that $|f_{\kappa(n)}(0)| \to\infty$, one has $\Delta(f_{\kappa(n)})\to\sup V$, which is impossible since $$\inf_{f,\L(f)\leq L}\Delta(f)\le E[V(|X|)]<\sup V.$$ So, the sequence $(f_k(0))_{n\geq 1}$ is bounded.

\subsection{Proof of Lemma \ref{lem:mesurcard}}\label{section:app-mes}

	\begin{enumerate}
		\item Let us first assume that $A$ is compact.
		For $k\geq 1$, we set
		$$N_{\theta,r}^k(A)=\inf_{\in \mathcal C}\sum_{S\in \mathcal C}\I_{\{D_{\theta,r}\cap S=\emptyset\}},$$ where the infimum is taken over all finite coverings 
		$\mathcal C$ of the compact set $A$ with open sets $S$ such that $\mbox{diam}(S)\leq \frac{1}{k}$. For every open set $S$, the function $(\theta,r)\mapsto \I_{\{D_{\theta,r}\cap S=\emptyset\}}$ is upper semi-continuous since $\{(\theta,r), D_{\theta,r}\cap S\neq\emptyset\}$
 is open. Thus, $(\theta,r)\mapsto N_{\theta,r}^k(A)$ is also upper semi-continuous. Moreover, $$N_{\theta,r}(A)=\lim_{k\to \infty}\uparrow N_{\theta,r}^k(A).$$ Consequently, $ (\theta,r)\mapsto N_{\theta,r}(A)$ is measurable.
 		\item According to the Cauchy formula, 
 		\begin{align*}
 		\L(f)&=\frac1{c_d}\int_{\mathcal S^{d-1}} \int_0^\infty \mbox{Card}(\{t\in[0,1],f(t)\in D_{\theta,r}\}) dr d\theta\\&=
 		\frac1{c_d}\int_{\mathcal S^{d-1}} \int_0^\infty 
 		\sum_{y\in \Im f \cap D_{\theta,r}}\mbox{Card}(f^{-1}(\{y\}))drd\theta.
 		\end{align*}
 		Since $\L(f)<\infty$, the set $B=\{ (\theta,r), N_{\theta,r}(\Im f)=\infty\}$ is $d\theta dr$ negligible.
 		Moreover, $\mathcal M=\{A\in \mathcal B(\R^d)\cap \Im f, (\theta,r)\mapsto N_{\theta,r}(A)\I_{B^c}(\theta,r)\mbox{ is measurable}\}$ is a monotone class. Hence, $\mathcal M= \mathcal B(\R^d)\cap \Im f$.
 		Thus, $(\theta,r)\mapsto N_{\theta,r}(A)$ is measurable for the Lebesgue sigma-algebra for every $A\subset  \mathcal B(\R^d)\cap \Im f$.
	\end{enumerate}

\bibliography{courbiblio}
\bibliographystyle{plainnat}

\end{document}